\documentclass[reqno]{amsart}

\usepackage[latin1]{inputenc}
\usepackage{amssymb}
\usepackage{graphicx}
\usepackage{amscd}
\usepackage[hidelinks]{hyperref}
\usepackage{color}
\usepackage{float}
\usepackage{graphics,amsmath,amssymb}
\usepackage{amsthm}
\usepackage{amsfonts}
\usepackage{latexsym}
\usepackage{epsf}
\usepackage{enumerate}
\usepackage{xifthen}
\usepackage{mathrsfs}
\usepackage{dsfont}
\usepackage{makecell}
\usepackage[FIGTOPCAP]{subfigure}
\usepackage{amsmath}
\allowdisplaybreaks[4]
\usepackage{listings}
\usepackage{etoolbox}
\usepackage{fancyhdr}

 \newtheoremstyle{mytheorem}
 {3pt}
 {3pt}
 {\slshape}
 {}
 {\bfseries}
 {.}
 { }
 {}

\numberwithin{equation}{section}

\theoremstyle{theorem}
\newtheorem{theorem}{Theorem}[section]
\newtheorem{corollary}[theorem]{Corollary}
\newtheorem{lemma}[theorem]{Lemma}

\theoremstyle{definition}

\newcommand{\arxiv}[1]{\href{http://arxiv.org/abs/#1}{arXiv:#1}}

\newcommand{\Keywords}[1]{\ifthenelse{\isempty{#1}}{}{\smallskip \smallskip \noindent \textbf{Keywords}. #1}}
\newcommand{\MSC}[2][2010]{\ifthenelse{\isempty{#2}}{}{\smallskip \smallskip \noindent \textbf{#1MSC}. #2}}
\newcommand{\abstractnote}[1]{\ifthenelse{\isempty{#1}}{}{\smallskip \smallskip \noindent \textsuperscript{\dag}#1}}

\makeatletter
\def\specialsection{\@startsection{section}{1}%
  \z@{\linespacing\@plus\linespacing}{.5\linespacing}%
  {\normalfont}}
\def\section{\@startsection{section}{1}%
  \z@{.7\linespacing\@plus\linespacing}{.5\linespacing}%
  {\normalfont\scshape}}
\patchcmd{\@settitle}{\uppercasenonmath\@title}{\Large\boldmath}{}{}
\patchcmd{\@settitle}{\begin{center}}{\begin{flushleft}}{}{}
\patchcmd{\@settitle}{\end{center}}{\end{flushleft}}{}{}
\patchcmd{\@setauthors}{\MakeUppercase}{\normalsize}{}{}
\patchcmd{\@setauthors}{\centering}{\raggedright}{}{}
\patchcmd{\section}{\scshape}{\large\bfseries\boldmath}{}{}
\patchcmd{\subsection}{\bfseries}{\bfseries\boldmath}{}{}
\renewcommand{\@secnumfont}{\bfseries}
\patchcmd{\@startsection}{\@afterindenttrue}{\@afterindentfalse}{}{}
\patchcmd{\abstract}{\leftmargin3pc}{\leftmargin1pc}{}{}

\def\maketitle{\par
  \@topnum\z@ 
  \@setcopyright
  \thispagestyle{empty}
  \ifx\@empty\shortauthors \let\shortauthors\shorttitle
  \else \andify\shortauthors
  \fi
  \@maketitle@hook
  \begingroup
  \@maketitle
  \toks@\@xp{\shortauthors}\@temptokena\@xp{\shorttitle}%
  \toks4{\def\\{ \ignorespaces}}
  \edef\@tempa{%
    \@nx\markboth{\the\toks4
      \@nx\MakeUppercase{\the\toks@}}{\the\@temptokena}}%
  \@tempa
  \endgroup
  \c@footnote\z@
  \@cleartopmattertags
}
\makeatother

\newcommand{\pt}{p_{3,3}}

\title[Congruences for $2$-color partition triples]{Ramanujan-type congruences for $2$-color partition triples}

\author[S. Chern]{Shane Chern}
\address[S. Chern]{Department of Mathematics, The Pennsylvania State University, University Park, PA 16802, USA}
\email{shanechern@psu.edu}

\author[C. Wang]{Chun Wang}
\address[C. Wang]{Department of Mathematics, East China Normal University, 500 Dongchuan Road, Shanghai 200241, PR China}
\email{wangchunmath@outlook.com}

\date{}

\begin{document}

{\footnotesize\noindent \textit{Preprint} (2017). Available at \arxiv{1706.05667}.}

\bigskip \bigskip

\maketitle

\begin{abstract}

Let ${p}_{3,3}(n)$ denote the number of $2$-color partition triples of $n$ where one of the colors appears only in parts that are multiples of $3$. In this paper, we shall establish some interesting Ramanujan-type congruences for ${p}_{3,3}(n)$.


\MSC{Primary 11P83; Secondary 05A17.}

\end{abstract}

\section{Introduction}

A \textit{partition} of a natural number $n$ is a weakly decreasing sequence of positive integers whose sum equals $n$. Let $p(n)$ be the number of partitions of $n$. We know that its generating function is
$$\sum_{n\ge 0}p(n)q^n=\frac{1}{(q;q)_\infty},$$
where for $|q|<1$, the shifted factorial is defined by
$$(a;q)_{\infty}:=\prod_{k\ge 0}(1-aq^k).$$

In 1919, Ramanujan \cite{Ram1919} discovered the following celebrated congruences
\begin{align*}
p(5n+4)&\equiv 0\pmod{5},\\
p(7n+5)&\equiv 0\pmod{7},\\
p(11n+6)&\equiv 0\pmod{11}.
\end{align*}
As an analogue of the ordinary partition function, Chan \cite{Cha2010} defined the cubic partition function $a(n)$ by
$$\sum_{n\ge 0}a(n)q^n:=\frac{1}{(q;q)_\infty(q^2;q^2)_\infty},$$
which enumerates the number of $2$-color partitions of $n$ where one of the colors appears only in multiples of $2$. He also established the partition congruence
$$a(3n+2)\equiv0\pmod{3}.$$
Subsequently, many authors studied the arithmetic properties of $2$-color partitions with
one of the colors appearing only in multiples of $k$; see \cite{ABD2015, BS2013, CL2009, Che2016, CD2016} for details.

Meanwhile, Chan and Cooper \cite{CC2010} studied the divisibility properties of the function $c(n)$ defined by
$$\sum_{n\ge 0}c(n)q^n:=\frac{1}{(q;q)^2_\infty(q^3;q^3)^2_\infty}$$
and obtained the following partition congruence
$$c(2n+1)\equiv0\pmod{2}.$$
Here the partition function $c(n)$ can be regarded as the number of $2$-color partition pairs of $n$ where one of the colors appears only in parts that are multiples of $3$. Moreover, by considering the generalized partition function $p_{[c^\ell d^m]}(n)$ defined by the generating function
$$\sum_{n\ge 0}p_{[c^\ell d^m]}(n)q^n
:=\frac{1}{(q^c;q^c)^\ell_\infty(q^d;q^d)^m_\infty}$$
and appealing to Ramanujan's modular equations, Baruah and Ojah \cite{BO2012} presented new proofs of several formulas obtained by Chan and Toh \cite{CT2010} and established more Ramanujan-type congruences, including $c(4n+3)\equiv 0\pmod{4}$.

Inspired by their work, we shall study the following 2-color partition triple function
\begin{equation}\label{eq1}
\sum_{n\ge 0}{{p}_{3,3}}(n)q^n:=\frac{1}{(q;q)^3_{\infty}(q^3;q^3)^3_{\infty}}.
\end{equation}

\begin{theorem}\label{th:1}
For $n\ge 0$, we have
\begin{gather}
\pt(12n+6,9)\equiv 0 \pmod{2},\label{eq:2-12}\\
\pt(6n+4)\equiv 0 \pmod{4},\label{eq:4-6}\\
\pt(3n+1)\equiv 0 \pmod{3},\label{eq:3-3}\\
\pt(3n+2)\equiv 0 \pmod{9},\label{eq:9-3}\\
\pt(9n+5,8)\equiv 0 \pmod{27},\label{eq:27-9}\\
\pt(5n+3)\equiv 0 \pmod{5}.\label{eq:5-5}
\end{gather}
\end{theorem}

\begin{theorem}\label{th:2}
For $n\ge 0$, $\alpha\ge 1$, and odd prime $p$ with
$$\left(\frac{-3}{p}\right)=-1,$$
we have
\begin{equation}\label{eq:p-cong}
\pt\left(9p^{2\alpha}n+\frac{p^{2\alpha-1}(3p+18j)+1}{2}\right)\equiv 0 \pmod{27},\\
\end{equation}
where $j=1$, $2$, $\ldots$, $p-1$.
\end{theorem}

\section{Preliminaries}

Throughout this paper, we write $f_k:=(q^k;q^k)_\infty$ for positive integers $k$ for notational convenience.

The following $2$-dissections are necessary.
\begin{lemma}
It holds that
\begin{align}
f_1f_3&=\frac{f_2 f_8^2 f_{12}^4}{f_4^2 f_6 f_{24}^2}-q\frac{f_4^4 f_6 f_{24}^2}{f_2 f_8^2 f_{12}^2},\label{f1f3}\\
\frac{f_1}{f_3}&=\frac{f_2f_{16}f_{24}^2}{f_6^2f_8f_{48}}-q\frac{f_2f_8^2f_{12}f_{48}}{f_4f_6^2f_{16}f_{24}},\label{f1df3}\\
\frac{f_3}{f_1}&=\frac{f_4f_6f_{16}f_{24}^2}{f_2^2f_8f_{12}f_{48}}+q\frac{f_6f_8^2f_{48}}{f_2^2f_{16}f_{24}}.\label{eq:f3df1}
\end{align}
\end{lemma}

\begin{proof}
Here \eqref{f1f3}, \eqref{f1df3} and \eqref{eq:f3df1} are respectively (30.12.1), (30.10.1) and (30.10.3) in \cite{Hir2017}.
\end{proof}

We also need the following $3$-dissection identities.

\begin{lemma}\label{le:3-dis}
It holds that
\begin{align}
f_1^3&=P(q^3)-3qf_9^3,\label{f13-3dis}\\
\frac{1}{f_1^3}&=\frac{f_9^3}{f_3^{12}}\Big(P(q^3)^2+3qf_9^3P(q^3)
+9q^2f_9^6\Big),\label{eq:1/f13--3dis}
\end{align}
where
\begin{align}
P(q)=\frac{f_2^6f_3}{f_1^2f_6^2}+3q\frac{f_1^2f_6^6}{f_2^2f_3^3}.\label{eq:Pf}
\end{align}
\end{lemma}

\begin{proof}
For \eqref{f13-3dis}, see \cite[Eq.~(21.3.3)]{Hir2017}. One may obtain \eqref{eq:1/f13--3dis} by replacing $q$ with $\omega q$ and $\omega^2 q$ in \eqref{f13-3dis} and multiplying the two results. Finally, \eqref{eq:Pf} follows from (21.3.7), (21.1.1) and (22.11.6) in \cite{Hir2017}.
\end{proof}

\begin{corollary}
It holds that
\begin{equation}\label{eq:f13diss}
\frac{1}{f_1^3}=\frac{f_{9}^3}{f_{3}^{12}}\Big(f_1^6+9qf_1^3 f_{9}^3+27q^2 f_{9}^6\Big).
\end{equation}
\end{corollary}

\begin{proof}
It follows by substituting $P(q^3)=f_1^3+3qf_9^3$ in \eqref{eq:1/f13--3dis}.
\end{proof}

At last, we recall the $p$-dissection formula of $f(-q):=(q;q)_\infty$.
\begin{lemma}[{\cite[Theorem 2.2]{CG2013}}]\label{cg13}
For any prime $p\geq5$,
\begin{align*}
f(-q)&=(-1)^{\frac{\pm p-1}{6}}q^{\frac{p^{2}-1}{24}}f(-q^{p^{2}})\\
&\quad\quad+\sum_{\substack{k=-\frac{p-1}{2}\\
k\neq\frac{\pm p-1}{6}}}^{\frac{p-1}{2}}(-1)^{k}q^{\frac{3k^{2}+k}{2}}f\left(-q^{\frac{3p^{2}+(6k+1)p}{2}},-q^{\frac{3p^{2}-(6k+1)p}{2}}\right).
\end{align*}
We further claim that for $-(p-1)/2\leq k\leq(p-1)/2$ and
$k\neq(\pm p-1)/6$,
\begin{align*}
\frac{3k^{2}+k}{2} \not\equiv
\frac{p^{2}-1}{24} \pmod{p}.
\end{align*}
Here for any prime $p\geq 5$,
\begin{equation*}
\frac{\pm p-1}{6}:=\left\{\begin{array}{ll}\frac{p-1}{6},&\ p \equiv 1
\pmod{6},\\[5pt]
\frac{-p-1}{6},&\ p \equiv -1 \pmod{6}.\end{array}\right.
\end{equation*}
\end{lemma}

\section{Proofs of Theorems \ref{th:1} and \ref{th:2}}

\begin{proof}[Proof of Theorem \ref{th:1}]
From \eqref{eq1}, we have
\begin{align*}
\sum_{n\ge 0}\pt(n)q^n=\frac{1}{f_1^3 f_3^3}\equiv \frac{f_1 f_3}{f_2^2 f_6^2}=\frac{1}{f_2^2 f_6^2}\left(\frac{f_2 f_8^2 f_{12}^4}{f_4^2 f_6 f_{24}^2}-q\frac{f_4^4 f_6 f_{24}^2}{f_2 f_8^2 f_{12}^2}\right) \pmod{4}.
\end{align*}
We now extract
\begin{align}
\sum_{n\ge 0}\pt(2n)q^n&\equiv \frac{1}{f_1^2 f_3^2}\cdot \frac{f_1 f_4^2 f_{6}^4}{f_2^2 f_3 f_{12}^2}\equiv  \frac{1}{f_1^2 f_3^2}\cdot \frac{f_1 f_1^8}{f_1^4 f_3}\notag\\
&=\frac{f_1^3}{f_3^3}=\frac{1}{f_3^3}\Big(P(q^3)-3qf_9^3\Big) \pmod{4}.
\end{align}
Since there are no terms in which the power of $q$ is $2$ modulo $3$, we arrive at \eqref{eq:4-6}.

On the other hand, we have
\begin{align}\label{eq:3ddd}
\sum_{n\ge 0}\pt(n)q^n=\frac{1}{f_1^3 f_3^3}=\frac{1}{f_3^3}\left(\frac{f_9^3}{f_3^{12}}\Big(P(q^3)^2+3qf_9^3P(q^3)
+9q^2f_9^6\Big)\right).
\end{align}
We extract
\begin{align}
\sum_{n\ge 0}\pt(3n+1)q^n=3\frac{f_3^6}{f_1^{15}}P(q).
\end{align}
This implies \eqref{eq:3-3}.

We also extract from \eqref{eq:3ddd} that
\begin{align}\label{eq:3n+2}
\sum_{n\ge 0}\pt(3n+2)q^n=9\frac{f_3^9}{f_1^{15}}.
\end{align}
This implies \eqref{eq:9-3}. We may further deduce from \eqref{eq:3n+2} that
\begin{align}\label{eq:3n+2mod27}
\sum_{n\ge 0}\pt(3n+2)q^n=9\frac{f_3^9}{f_1^{15}}\equiv 9f_3^4 \pmod{27}.
\end{align}
Since there are no terms on the right in which the power of $q$ is $1$ or $2$ modulo $3$, we obtain \eqref{eq:27-9}.

Furthermore, we deduce from \eqref{eq:3ddd} that
\begin{align}\label{eq:3n}
\sum_{n\ge 0}\pt(3n)q^n&=\frac{f_3^3}{f_1^{15}}P(q)^2=\frac{f_3^3}{f_1^{15}}\left(\frac{f_2^6f_3}{f_1^2f_6^2}+3q\frac{f_1^2f_6^6}{f_2^2f_3^3}\right)^2\notag\\
&\equiv \frac{f_2^{12} f_3^5}{f_1^{19} f_6^4}+q^2 \frac{f_6^{12}}{f_1^{11}f_2^4 f_3^3}\notag\\
&\equiv f_1 f_3\left(\frac{f_4}{f_{12}}+q^2 \frac{f_{12}^5}{f_4^5}\right)\notag\\
&=\left(\frac{f_2 f_8^2 f_{12}^4}{f_4^2 f_6 f_{24}^2}-q\frac{f_4^4 f_6 f_{24}^2}{f_2 f_8^2 f_{12}^2}\right)\left(\frac{f_4}{f_{12}}+q^2 \frac{f_{12}^5}{f_4^5}\right) \pmod{2}.
\end{align}
We extract
\begin{align}
\sum_{n\ge 0}\pt(6n)q^n&\equiv\frac{f_1 f_4^2 f_{6}^4}{f_2^2 f_3 f_{12}^2}\left(\frac{f_2}{f_{6}}+q \frac{f_{6}^5}{f_2^5}\right)=\frac{f_1}{f_3}\cdot \frac{f_4^2 f_{6}^4}{f_2^2 f_{12}^2}\left(\frac{f_2}{f_{6}}+q \frac{f_{6}^5}{f_2^5}\right)\notag\\
&=\left(\frac{f_2f_{16}f_{24}^2}{f_6^2f_8f_{48}}-q\frac{f_2f_8^2f_{12}f_{48}}{f_4f_6^2f_{16}f_{24}}\right)\cdot \frac{f_4^2 f_{6}^4}{f_2^2 f_{12}^2}\left(\frac{f_2}{f_{6}}+q \frac{f_{6}^5}{f_2^5}\right)\pmod{2}.
\end{align}
Hence
\begin{align}
\sum_{n\ge 0}\pt(12n+6)q^n&\equiv\frac{f_2^2 f_3^7 f_{8} f_{12}^2}{f_1^6 f_4 f_{6}^2 f_{24}}+\frac{f_2 f_3 f_4^2 f_{24}}{f_{6} f_{8} f_{12}}\notag\\
&\equiv \frac{f_1^4 f_3^{7} f_{1}^{8} f_{3}^{8}}{f_1^{6} f_1^4 f_{3}^4 f_{3}^{8}}+\frac{f_1^2 f_3 f_1^{8} f_{3}^{8}}{f_{3}^2 f_1^{8} f_{3}^4}\notag\\
&= f_1^2 f_3^3+f_1^2 f_3^3\equiv 0\pmod{2}.
\end{align}
Hence $\pt(12n+6)\equiv 0 \pmod{2}$.

We may also extract from \eqref{eq:3n}
\begin{align}
\sum_{n\ge 0}\pt(6n+3)q^n&\equiv\frac{f_2^4 f_3 f_{12}^2}{f_1 f_4^2 f_{6}^2}\left(\frac{f_2}{f_{6}}+q \frac{f_{6}^5}{f_2^5}\right)=\frac{f_3}{f_1}\cdot \frac{f_2^4 f_{12}^2}{f_4^2 f_{6}^2}\left(\frac{f_2}{f_{6}}+q \frac{f_{6}^5}{f_2^5}\right)\notag\\
&=\left(\frac{f_4f_6f_{16}f_{24}^2}{f_2^2f_8f_{12}f_{48}}+q\frac{f_6f_8^2f_{48}}{f_2^2f_{16}f_{24}}\right)\cdot \frac{f_2^4 f_{12}^2}{f_4^2 f_{6}^2}\left(\frac{f_2}{f_{6}}+q \frac{f_{6}^5}{f_2^5}\right)\pmod{2}.
\end{align}
Hence
\begin{align}
\sum_{n\ge 0}\pt(12n+9)q^n&\equiv\frac{f_3^4 f_6 f_8 f_{12}^2}{f_1^3 f_2 f_4 f_{24}}+\frac{f_1^3 f_4^2 f_6^2 f_{24}}{f_2^2 f_3^2 f_8 f_{12}}\notag\\
&\equiv \frac{f_3^4 f_3^2 f_1^8 f_3^8}{f_1^3 f_1^2 f_1^4 f_3^8}+\frac{f_1^3 f_1^8 f_3^4 f_3^8}{f_1^4 f_3^2 f_1^8 f_3^4}\notag\\
&=\frac{f_3^6}{f_1}+\frac{f_3^6}{f_1}\equiv 0\pmod{2}.
\end{align}
Hence $\pt(12n+9)\equiv 0 \pmod{2}$.

At last, we show \eqref{eq:5-5}. It follows from \eqref{eq1} and \eqref{eq:f13diss} that
\begin{align}
\sum_{n\ge 0}\pt(n)q^n&=\frac{1}{f_1^3 f_3^3}=\frac{1}{f_3^3}\left(\frac{f_{9}^3}{f_{3}^{12}}\Big(f_1^6+9qf_1^3 f_{9}^3+27q^2 f_{9}^6\Big)\right)\notag\\
&=\frac{f_1^6 f_9^3}{f_3^{15}}+9q\frac{f_1^3 f_9^6}{f_3^{15}}+27q^2 \frac{f_9^9}{f_3^{15}}\notag\\
&\equiv \frac{f_5}{f_{15}^3} f_1f_9^3+4q\frac{f_{45}}{f_{15}^3}f_1^3 f_9+2q^2 \frac{f_{45}^2}{f_{15}^3}\frac{1}{f_9}\notag\\
&\equiv \frac{f_5}{f_{15}^3}(E_0+E_1+E_2)(J_0^*+J_4^*)+4q\frac{f_{45}}{f_{15}^3}(J_0+J_1)(E_0^*+E_4^*+E_3^*)\notag\\
&\qquad+2q^2 \frac{f_{45}^2}{f_{15}^3}(P_0^*+P_4^*+P_3^*+P_2^*) \pmod{5}.
\end{align}
Here, $S_k$ and $S_k^*$ indicate series in which the powers of $q$ are congruent to $k$ modulo
$5$, whether $S$ is $E$ (for Euler), $J$ (for Jacobi) or $P$ (for partitions). Since there are no terms in which the power of $q$ is congruent to $3$ modulo $5$, we arrive at \eqref{eq:5-5}. We remark that the same technique is used in \cite[\S 36.4]{Hir2017}.
\end{proof}

\begin{proof}[Proof of Theorem \ref{th:2}]
We know from \eqref{eq:3n+2mod27} that
\begin{align*}
\sum_{n\ge 0}\pt(3n+2)q^n\equiv 9f_3^4 \equiv 9f_3 f_9\pmod{27}.
\end{align*}
We therefore extract
\begin{align}
\sum_{n\ge 0}\pt(9n+2)q^n\equiv 9f_1 f_3\pmod{27}.
\end{align}

Given a prime $p\geq5$, and integers $k$ and $m$ with $-(p-1)/2\leq k,m\leq(p-1)/2$, we consider the following quadratic congruence:
\begin{align*}
\frac{3k^{2}+k}{2}+3\cdot\frac{3m^{2}+m}{2}\equiv\frac{p^{2}-1}{6}\pmod p,
\end{align*}
that is,
\begin{align}\label{6k6m}
2(6k+1)^{2}+6(6m+1)^{2}\equiv0\pmod p.
\end{align}
We conclude that, for any odd prime $p$ with
$$\left(\frac{-3}{p}\right)=-1,$$
the solution to \eqref{6k6m} is $k=m=(\pm p-1)/6$.

It follows from Lemma \ref{cg13} that
\begin{align*}
\sum_{n\ge 0}\pt\left(9\left(pn+\frac{p^2-1}{6}\right)+2\right)q^n
\equiv9f(-q^p)f(-q^{3p})\pmod{27},
\end{align*}
and
\begin{align*}
\sum_{n\ge 0}\pt\left(9\left(p^2n+\frac{p^2-1}{6}\right)+2\right)q^n
\equiv9f(-q)f(-q^{3})\pmod{27}.
\end{align*}

At last, we induct on $\alpha\ge 1$ to obtain
\begin{align*}
\sum_{n\ge 0}\pt\left(9p^{2\alpha-1}n+\frac{3p^{2\alpha}+1}{2}\right)q^n
\equiv9f(-q^p)f(-q^{3p})\pmod{27}.
\end{align*}
This implies that
\begin{align*}
\pt\left(9p^{2\alpha-1}(pn+j)+\frac{3p^{2\alpha}+1}{2}\right)\equiv 0 \pmod{27},
\end{align*}
where $j=1$, $2$, $\ldots$, $p-1$. We arrive at \eqref{eq:p-cong}.
\end{proof}

\section{Final remarks}

Using an algorithm (which involves modular forms) due to Radu and Sellers \cite{Rad2009,RS2011}, we are able to prove the following congruences modulo $7$ and $11$:

\begin{theorem}
For $n\ge 0$, we have
\begin{gather}
\pt(21n + 7,10,16,18)\equiv 0 \pmod{7},\\
\pt(121n + 39, 61, 72, 94, 105, 116)\equiv 0 \pmod{11}.
\end{gather}
\end{theorem}

However, it is still unclear if there exist any elementary proofs of these congruences.

\subsection*{Acknowledgements}

C.~Wang was partially supported by the outstanding doctoral dissertation cultivation plan of action (No.~YB2016028).

\bibliographystyle{amsplain}

\end{document}